\documentclass{amsart}
\usepackage{amssymb}
\usepackage{amscd}
\usepackage{amsfonts}
\usepackage{latexsym}
\usepackage[all]{xy}
\usepackage{hyperref}
\usepackage{verbatim}

\newtheorem{theorem}{Theorem}[section]
\newtheorem{lemma}[theorem]{Lemma}
\newtheorem{proposition}[theorem]{Proposition}
\newtheorem{corollary}[theorem]{Corollary} 
\theoremstyle{definition}  
\newtheorem{definition}[theorem]{Definition}
\newtheorem{example}[theorem]{Example}
\newtheorem{conjecture}[theorem]{Conjecture}  
\newtheorem{question}[theorem]{Question}
\newtheorem{remark}[theorem]{Remark}

\newcommand{\Tr}{\text{Tr}}
\newcommand{\id}{\text{id}}

\newcommand{\Ver}{\text{Ver}}

\newcommand{\Fr}{\text{Fr}}
\newcommand{\FPdim}{\text{FPdim}} 
\newcommand{\End}{\text{End}} 
 
\renewcommand{\Vec}{\text{Vec}}
\newcommand{\sVec}{\text{sVec}}

\newcommand{\Hom}{\text{Hom}} 
 
\newcommand{\uRep}{\underline{\text{Rep}}}

\newcommand{\Rep}{\text{Rep}}

\newcommand{\ev}{\text{ev}}
\newcommand{\coev}{\text{coev}}

\newcommand{\eps}{\varepsilon}

\newcommand{\bL}{{\bar L}}

\newcommand{\C}{\mathcal{C}}
\newcommand{\D}{\mathcal{D}}

\newcommand{\F}{\mathbb{F}}

\renewcommand{\L}{\mathcal{L}}

\newcommand{\A}{\mathcal{A}}

\newcommand{\N}{\mathcal{N}}

\renewcommand{\O}{\mathcal{O}}

\newcommand{\be}{\mathbf{1}}

\renewcommand{\k}{\mathbf{k}}

\renewcommand{\be}{\mathbf{1}}

\newcommand{\T}{{\mathcal T}}
\newcommand{\BZ}{{\mathbb Z}}

\newcommand{\BR}{{\mathbb R}}

\newcommand{\bt}{\boxtimes}

\newcommand{\ot}{\otimes}

\begin{document}

\title{On symmetric fusion categories in positive characteristic}

\author{Victor Ostrik}
\address{V.O.: Department of Mathematics,
University of Oregon, Eugene, OR 97403, USA}
\email{vostrik@math.uoregon.edu}

\begin{abstract} We propose a conjectural extension to positive characteristic case 
of a well known Deligne's theorem on the existence of super fiber functors. We prove
our conjecture in the special case of semisimple categories with finitely many
isomorphism classes of simple objects. 
\end{abstract}

\date{\today} 
\maketitle  

\section{Introduction}
\subsection{}
Let $\k$ be an algebraically closed field of characteristic $p\ge 0$. 
We recall that a {\em symmetric tensor category} $\C$ is a category
endowed with the functor $\ot :\C \times \C \to \C$ of tensor product,
and with associativity and commutativity isomorphisms and unit object $\be$
satisfying suitable axioms, see e.g. \cite{SR} or \cite{EGNO}. In this paper we consider symmetric
tensor categories $\C$ satisfying the following assumptions:

1) $\C$ is an essentially small $\k-$linear abelian category such that any morphism
space is finite dimensional and each object has finite length;

2) the functor $\ot$ is $\k-$linear and the natural morphism $\k \to \End(\be)$ to the endomorphism ring of the unit object is an isomorphism;

3) $\C$ is rigid (this implies that the functor $\ot$ is exact in each variable, see \cite[Proposition 1.16]{DM}).

Such categories are precisely tensor categories satisfying finiteness assumptions
of \cite[2.12.1]{Dta}; they were called {\em pre-Tannakian} in \cite[2.1]{CO}. 


\begin{example} \label{gexa}
(i) The category $\Vec$ of finite dimensional vector spaces is pre-Tannakian.
Now let $p\ne 2$. Then the category $\sVec$ of finite dimensional super vector spaces
over $\k$ is pre-Tannakian. 

(ii) Let $G$ be an affine group scheme over $\k$. Then the category
$\Rep_\k(G)$ of finite dimensional representations of $G$ over $\k$ is
pre-Tannakian.

(iii) (see \cite[0.3]{Dte}) 
Let $G$ be an affine super group scheme over $\k$ and let $\eps \in G(\k )$ be an
element of order $\le 2$ such that its action by conjugation on $G$ coincides with the
parity automorphism of $G$. Let $\Rep_\k(G,\eps)$ be the full subcategory of super representations
of $G$ such that $\eps$ acts by parity automorphism. Then $\Rep_\k(G,\eps)$ is
pre-Tannakian. A special case of this construction is when $G$ is a finite group and
$\eps \in G$ is a central element of order $\le 2$, see \cite[0.4 (i)]{Dte}.

(iv) (see \cite[Section 8]{Dta}) Let $\C$ be a pre-Tannakian category and let $\pi \in \C$ be its
{\em fundamental group} as defined in \cite[8.13]{Dta}. Thus $\pi$ is an affine groups scheme in
the category $\C$ and it acts on any object of $\C$ in 
a canonical way. Let $G$ be an affine group scheme in the category $\C$ and let 
$\eps :\pi \to G$ be a homomorphism such that the action of $\pi$ on $G$ by conjugations
coincides with the canonical action. Let $\Rep_\C(G)$ be the category of representations of 
$G$ in category $\C$ and let $\Rep_\C(G,\eps)$ be the full subcategory of $\Rep_\C(G)$
consisting of representations such that the action of $\pi$ via homomorphism $\eps$
coincides with the canonical action of $\pi$. Then both $\Rep_\C(G)$ and $\Rep_\C(G,\eps)$
are pre-Tannakian. Example (iii) is a special case of this with $\C =\sVec$ since the
fundamental group of $\sVec$ is the finite group of order 2 and its canonical action is 
given by the parity automorphism.
\end{example}

In this paper a {\em symmetric tensor functor} between symmetric tensor categories is a monoidal
functor compatible with the commutativity isomorphism. Recall (see \cite{SR}) that a {\em fiber functor} for
a pre-Tannakian category $\C$ is a $\k-$linear exact symmetric tensor functor $\C \to \Vec$. Similalrly, a {\em super fiber functor}
is a $\k-$linear exact symmetric tensor functor $\C \to \sVec$, see \cite{Dte}.

For example the forgetful functor $\Rep_\k(G)\to \Vec$ assigning to a representation its underlying vector
space is a fiber functor. Similarly, the forgetful functor $\Rep_\k(G,\eps)\to \sVec$ is
a super fiber functor. Conversely, in the theory of {\em Tannakian categories} (see \cite{SR, DM, Dta}) one shows
that any pre-Tannakian category $\C$ with a fiber functor is tensor equivalent to $\Rep_\k(G)$ endowed with
the forgetful functor. Similarly and more generally, any  pre-Tannakian category $\C$ with a super fiber functor is tensor equivalent 
to $\Rep_\k(G, \eps)$ endowed with the forgetful functor, see \cite[8.19]{Dta}. These results reduce many questions 
about pre-Tannakian categories to the theory of affine group schemes and super schemes.

Furthermore, Deligne showed that for $p=0$ many pre-Tannakian categories admit a super fiber functor. Namely,
we say that $\C$ is of {\em subexponential growth} if for any object $X\in \C$ the length of the objects $X^{\ot n}$ is
bounded by the function $a_X^n$ for a suitable $a_X\in \BR$, see \cite{Dte}, \cite[9.11]{EGNO}.

\begin{theorem}[\cite{Dte} Th\'eor\`eme 0.6]\label{p=0}
 Assume that $p=0$. A pre-Tannakian category $\C$ admits a super fiber functor if and only if
it is of subexponential growth.
\end{theorem}

\subsection{}\label{conjsec}
The main goal of this paper is to propose a conjectural extension of Theorem \ref{p=0} to the case $p\ne 0$.
The counterexamples constructed by Gelfand and Kazhdan in \cite{GK} (see also \cite{A,GM}) show that a direct 
counterpart of Theorem \ref{p=0} fails for $p>0$. For instance for $p=5$ there exists a semisimple pre-Tannakian
category $\C$, called {\em Yang-Lee category} or {\em Fibbonacci category}, 
with two isomorphism classes $\be$ and $X$ of simple objects and such that $X\ot X=\be \oplus X$,
see \cite{GK, GM} and Example \ref{lexam} below.
It is clear that this category has no super fiber functors since for any monoidal functor $F: \C \to \sVec$ the dimension $d$
of vector space $F(X)$ would be a root of the equation $d^2=1+d$ which is impossible. 

Thus in Section \ref{Frobe} for each prime $p$ we introduce the universal Verlinde category $\Ver_p$ which is a 
semisimple pre-Tannakian
category with $p-1$ isomorphism classes of simple objects (this category is equivalent to 
one of the categories constructed in \cite{GK, GM}, see Section \ref{sl2}).

\begin{conjecture}\label{mconj}
 Assume that $p>0$. A pre-Tannakian category $\C$ of subexponential growth admits a 
(unique up to isomorphism) $\k-$linear exact symmetric tensor functor $\C \to \Ver_p$. 
\end{conjecture}

\begin{remark} We refer the reader to \cite{Ds} for examples of pre-Tannakian categories 
which are not of subexponential growth.
We note that no such examples are currently known in the case $p>0$.
\end{remark}

In view of Example \ref{lexam} Conjecture \ref{mconj} states that in the case $p=2$ any pre-Tannakian category of subexponential
growth admits a fiber functor and in the case $p=3$ any pre-Tannakian category of subexponential growth admits a super fiber
functor. Thus Conjecture \ref{mconj} predicts that for $p=2$  any pre-Tannakian category of subexponential
growth is of the form $\Rep_\k(G)$ for a suitable affine group scheme $G$,  and for $p=3$ any pre-Tannakian category of subexponential growth is of the form $\Rep_\k(G,\eps)$ for a suitable affine super group scheme $G$.

\subsection{} We recall (see \cite{ENO}) that a {\em fusion category} is a $\k-$linear semisimple
rigid monoidal category with finite dimensional $\Hom-$spaces, finitely many isomorphism classes of simple objects, and
simple unit object. In particular, a symmetric fusion category (that is a fusion category equipped with a  
symmetric braiding) is the same as semisimple pre-Tannakian category with finitely many isomorphism 
classes of simple objects. It is not difficult to see that a fusion category is of subexponential growth,
see \cite[Lemme 4.8]{Dte}.
Thus the following statement which is the main result of this paper is a special case of
Conjecture \ref{mconj}:

\begin{theorem}\label{main}
 Let $p>0$. A symmetric fusion category $\C$ admits a $\k-$linear symmetric tensor functor $\C \to \Ver_p$.
\end{theorem}

We note that Theorem \ref{main} holds true also for $p=0$ if we set $Ver_0=\sVec$ by Theorem \ref{p=0}. Using \cite[Th\'eor\`eme 8.17]{Dta} we get the following

\begin{corollary}\label{grsch}
 A symmetric fusion category is of the form $\Rep_{\Ver_p}(G,\eps)$ where $G$ 
 is a finite group scheme in the category $\Ver_p$.
\end{corollary}

A well known {\em Nagata's theorem} (\cite[IV, 3.6]{DG}) gives a classification of finite group schemes $G$ such that $\Rep_\k(G)$
is semisimple; thus Corollary \ref{grsch} yields a classification of symmetric fusion categories in the case $p=2$. Namely, any such category is an equivariantization (see \cite[Section 4]{DGNO}) 
of a pointed category associated with a 2-group (see e.g. \cite[8.4]{EGNO}) by the action of 
a group of odd order. It is natural to ask

\begin{question} What is classification of finite group schemes $G$ in $\Ver_p$ such that 
$\Rep_{\Ver_p}(G)$ or $\Rep_{\Ver_p}(G,\eps)$ is semisimple?
\end{question}

\subsection{} The main ingredient in the proof of Theorem \ref{main} is the notion of {\em Frobenius functor}
which is an abstract version of the pullback functor under the Frobenius morphism from a group scheme to itself.
The definition of this functor is given in Section \ref{Frobe}, and it works only in the case of semisimple pre-Tannakian categories.
We expect that a similar definition can be given more generally and hope to address this issue in future publications.

Another essential tool in the proof of Theorem \ref{main} is the theory of 
{\em non-degenerate fusion categories} developed in \cite[Section 9]{ENO}. A crucial
property of such categories is that they can be lifted to characteristic zero, see {\em loc. cit.}

\subsection{Acknowledgements} It is my great pleasure to thank Pierre Deligne, 
Pavel Etingof, Michael
Finkelberg, Shlomo Gelaki, Alexander Kleshchev, Dmitri Nikshych, Julia Pevtsova, Alexander Polishchuk, and Vadim Vologodsky for very useful conversations.

\section{Preliminaries}
For a tensor category $\C$ we will denote by $\be$ its unit object. For a braided
(in particular, symmetric) tensor category $\C$ we will denote by $c_{X,Y}$ the braiding
morphism $X\ot Y\to Y\ot X$.
For an abelian category $\C$ we will denote by $\O(\C)$
the set of isomorphism classes of simple objects of $\C$.

\subsection{Fusion categories}\label{fgen}
 The definition of fusion category was given in introduction. 
A {\em fusion subcategory} of a fusion category $\C$ is  a full tensor subcategory $\C'\subset \C$
such that if $X\in \C$ is isomorphic to a direct summand of an object of $\C'$ then $X\in \C'$,
see \cite[2.1]{DGNO}. For a collection $S$ of objects of $\C$ there is a smallest fusion subcategory containing
$S$; it is called a fusion subcategory generated by $S$.

A tensor functor $F$ between fusion categories $\C$ and $\D$ is called {\em injective} if it is fully faithful; such a functor is called {\em surjective} if any object of $\D$ is isomorphic to a direct summand of $F(X), X\in \C$, see \cite[5.7]{ENO}. Thus a tensor functor is an equivalence 
if and only if it is both injective and surjective.

For a tensor functor $F: \C \to \D$ its {\em image} $F(\C)$
is the fusion subcategory of $\D$ generated by objects $F(X), X\in \C$. 

Let $G$ be a finite group. A {\em $G-$grading} on a fusion category $\C$ is a function
$\phi : \O(\C)\to G$ such that for $X,Y\in \O(\C)$ the tensor product $X \ot Y$ contains
only simple summands $Z$ with $\phi(Z)=\phi(X)\phi(Y)$, see e.g \cite[2.3]{DGNO}; 
such a grading is {\em faithful}
if the function $\phi$ is surjective. It is clear that direct sums of simple objects $X$ with
$\phi(X)=1\in G$ form a fusion subcategory of $\C$; this is {\em neutral component} of the
grading.

\subsection{External tensor product} Let $\C$ and $\D$ be two $\k-$linear
tensor categories. We define category $\C \times_\k \D$ as follows: objects
are pairs $(X,Y)$ where $X\in \C$ and $Y\in \D$ and morphisms are
$\Hom((X_1,Y_1),(X_2,Y_2))=\Hom_\C(X_1,X_2)\ot_\k \Hom(Y_1,Y_2)$.
The category $\C \times_\k \D$ has an obvious structure of $\k-$linear tensor category
with tensor product given by $(X_1,Y_1)\ot (X_2,Y_2):=(X_1\ot X_2,Y_1\ot Y_2)$.
If $\C$ and $\D$ are symmetric tensor categories the so is $\C \times_\k \D$.

We define {\em external tensor product} $\C \bt \D$ to be the {\em Karoubian
envelope} (see e.g. \cite[1.8]{Ds})  of $\C \times_\k \D$; the image of pair $(X,Y)\in \C \times_\k \D$ in
$\C \bt \D$ will be denoted $X\bt Y$. We have an obvious tensor functors $\C \to \C \bt \D$,
$X\mapsto X\bt \be$ and $\D \to \C \bt \D$, $Y\mapsto \be \bt Y$. If $\C, \D, \A$ are symmetric
$\k-$linear Karoubian tensor categories we have the following universal property of the category
$\C \bt \D$:

(a) the functor assigning to $F: \C \bt \D \to \A$ its composition with functors $\C \to \C \bt \D$
and $\D \to \C \bt \D$ is an equivalence of categories:

\{ $\k-$linear symmetric tensor functors $\C \bt \D \to \A$\} $\to$ \{ pairs of $\k-$linear 
symmetric tensor functors $\C \to \A$ and $\D \to \A$\}

In general the category $\C \bt \D$ is not abelian even if $\C$ and $\D$ are. However if 
$\C$ and $\D$ are abelian and one of these categories is semisimple then $\C \bt \D$
is also abelian (say if $\C$ is semisimple then $\C \bt \D$ is equivalent to direct sum
of copies of $\D$ indexed by the isomorphism classes of simple objects in $\C$ as an
additive category).  

\begin{example} \label{equi}
Let $\C$ be a semisimple pre-Tannakian category and let $G$ be a finite group.
Let $\C_G$ be the equivariantization of $\C$ with respect to the trivial action of $G$ on $\C$,
see \cite[4.1.3]{DGNO}. In other words the objects of $\C_G$ are objects of $\C$ equipped with
$G-$action; the morphisms are morphisms in $\C$ commuting with $G-$action, and the tensor
product is obvious. We have the following symmetric tensor functors:
$$\C \to \C_G, X\mapsto (X,\; \mbox{trivial action of}\; G),$$
$$\Rep_k(G)\to \C_G, V\mapsto (V\ot \be , G \; \mbox{acts on first factor}).$$ 
Thus by the universal property (a) we have a symmetric tensor functor
$\C \bt \Rep_\k (G)\to \C_G$. We leave it to the reader to check that this functor
is an equivalence.
\end{example}

For the future use we will record the following result:

\begin{lemma} \label{AA}
Let $\C$ be a symmetric fusion category and let $\A_1,\A_2\subset \C$ be two fusion subcategories.
Assume that the only simple object $X\in \C$ satisfying $X\in \A_1$ and $X\in \A_2$ is $X=\be$.
Then fusion subcategory $\langle \A_1,\A_2\rangle$ of $\C$ generated by (the objects of) $\A_1$ and $\A_2$ is
equivalent to $\A_1\bt \A_2$ as a symmetric tensor category.
\end{lemma}

\begin{proof} We have obvious symmetric tensor functors $\A_1, \A_2\to \langle \A_1,\A_2\rangle$;
thus by the universal property (a) we have a symmetric tensor functor $\A_1\bt \A_2\to
\langle \A_1,\A_2\rangle$ sending $X\bt Y$ to $X\ot Y$. This functor is clearly surjective, so we
just need to show that it is injective. Any simple object of $\A_1\bt \A_2$ is of the form
$X\bt Y$ where $X\in \O(\A_1)$ and $Y\in \O(\A_2)$. For two such objects $X_1\bt Y_1$ 
and $X_2\bt Y_2$ we have
$$\Hom(X_1\ot Y_1, X_2\ot Y_2)=\Hom(X_1\ot X_2^*, Y_1^*\ot Y_2)=\left\{ 
\begin{array}{cc} \k &\mbox{if}\; X_1=X_2\; \mbox{and}\; Y_1=Y_2\\ 0&\mbox{otherwise}
\end{array}\right.
$$
since $X_1\ot X_2^*\in \A_1$ and $Y_1^*\ot Y_2\in \A_2$ which proves the injectivity.
\end{proof}

\begin{remark} Lemma \ref{AA} extends trivially to the case when $\C$ is a semisimple 
pre-Tannakian category. On the other hand the condition that $\C$ is symmetric can not
be dropped. Namely if $\C$ is braided we still have an equivalence 
$\A_1\bt \A_2\simeq \langle \A_1,\A_2\rangle$ of monoidal categories but it is not
necessarily braided. If $\C$ is not braided then even a functor $\A_1\bt \A_2\to \langle \A_1,\A_2\rangle$ can not be defined in general.
\end{remark}

\subsection{Frobenius-Perron dimension}\label{FPdim}
For an abelian tensor category $\C$ with exact tensor product we will denote by $K(\C)$ its Grothendieck ring, see e.g.
\cite[4.5]{EGNO}. Class of an object $X\in \C$ in $K(\C)$ will be denoted $[X]$. We recall that for a fusion category $\C$
there is a unique ring homomorphism $\FPdim: K(\C)\to \BR$ called {\em Frobenius-Perron dimension} such that
$\FPdim(X):=\FPdim([X])>0$ for any $0\ne X\in \C$, see \cite[4.5]{EGNO}. This definition implies that $\FPdim(X)\ge 1$ for
any $X\ne 0$, see \cite[Proposition 3.3.4]{EGNO}.

Recall that {\em Frobenius-Perron dimension} $\FPdim(\C)$ of a fusion category $\C$ is defined as
$$\FPdim(\C)=\sum_{X\in \O(\C)}\FPdim(X)^2.$$

It is easy to see that for any $M\in \BR$ the set 
$$\{ x\in \BR | x<M\; \mbox{and there exists a fusion category}\; \C \; \mbox{with}\; x=\FPdim(\C)\}$$
is finite. In particular any nonempty set of fusion categories has an element $\C$ with minimal possible $\FPdim(\C)$.

We have the following result:

\begin{lemma}[\cite{EGNO} Propositions 6.3.3 and 6.3.4] Let $F:\C \to \D$ be a tensor functor between fusion categories.

(i) If $F$ is injective then $\FPdim(\C)\le \FPdim(\D)$ and we have equality if and only if $F$ is an equivalence;

(ii) If $F$ is surjective then $\FPdim(\C)\ge \FPdim(\D)$  and we have equality if and only if $F$ is an equivalence.
\end{lemma}

Since the functor $F: \C \to F(\C)$ is surjective we have the following

\begin{corollary} \label{mininj}
For a tensor functor $F:\C \to \D$ between fusion categories we have 
$\FPdim(F(\C))\le \FPdim(\C)$ and we have equality if and only if $F$ is injective.
\end{corollary}

\subsection{Non-degenerate fusion categories}\label{ndeg}
Let $\C$ be a $\k-$linear rigid tensor category such that $\k \to \End(\be)$ is
an isomorphism.
We recall that a {\em pivotal structure} on $\C$ is a functorial tensor
isomorphism $X\simeq X^{**}$ for any $X\in \C$, see \cite{BW} or \cite[4.7]{EGNO}. 
Such a structure allows to define the left
and right traces of any morphism $a: X\to X$, see {\em loc. cit}. A pivotal structure is called 
{\em spherical} if for any morphism $a: X\to X$ its left trace equals right trace, so the notion
of trace is unambiguous. In particular we can defined {\em dimension} $\dim(X)\in \k$ of any object 
$X$ as a trace of the identity morphism. If $\C$ is abelian the dimension determines a ring homomorphism
$\dim: K(\C)\to \k$ sending $[X]$ to $\dim(X)$. In particular this discussion applies in the case
when $\C$ is symmetric, since for such categories we have a canonical choice of spherical
structure given by
$$X \xrightarrow{\id_X\ot \coev_{X^*}} X\ot X^*\ot X^{**}\xrightarrow{c_{X,X^*}\ot \id_{X^{**}}} 
X^*\ot X\ot X^{**}\xrightarrow{\ev_X\ot \id_{X^{**}}} X^{**},$$
see e.g. \cite[Section 9.9]{EGNO}. This is the only spherical structure that is used in this paper.
Recall that (see e.g. \cite[Proposition 4.8.4]{EGNO}:

(a) if $\C$ is semisimple and $X\in \O(\C)$ then $\dim(X)\ne 0$.

A {\em spherical fusion category} is a fusion category equipped with a spherical structure.
For such category $\C$ one defines its {\em global dimension} $\dim(\C)\in \k$ via
$$\dim(\C)=\sum_{X\in \O(\C)}\dim(X)^2.$$

\begin{definition}[\cite{ENO} Definition 9.1] A spherical fusion category $\C$ is called {\em non-degenerate} if $\dim(\C)\ne 0$.
\end{definition}

\begin{remark} (i) In fact $\dim(\C)$ is independent of the choice of spherical structure.
Moreover, $\dim(\C)$ and the notion of non-degeneracy can be defined for a fusion
category without a reference to the spherical structures, see \cite[Definition 2.2]{ENO}.

(ii) It is known that for $p=0$ any fusion category is non-degenerate, see \cite[Theorem 2.3]{ENO}.
Thus this notion is of interest only for $p>0$.
\end{remark}

A crucial property of non-degenerate fusion categories is that they can be lifted to characteristic
zero, see \cite[Section 9]{ENO}. In particular we have the following

\begin{proposition}[\cite{EG} Theorem 5.4] \label{ndgfib}
Let $\C$ be a non-degenerate symmetric fusion category. 

(i) If $p=2$ then there exists a fiber functor $\C \to \Vec$;

(ii) If $p>2$ then there exists a super fiber functor $\C \to \sVec$.
\end{proposition}

\begin{proof} Let $W(\k)$ be the ring of Witt vectors of $\k$ and let $\F$ be its field of quotients.
Thus we have ring homomorphisms $W(\k)\to \k$ and $W(\k)\to \F$. 

By \cite[Corollary 9.4]{ENO} the category $\C$ has a lifting $\C_{W(\k)}$ to characteristic zero.
Thus $\C_{W(\k)}$ is a symmetric tensor category over $W(\k)$; its objects are the same as
objects of $\C$ and its morphisms are free $W(\k)-$modules and we have that 
$\C_{W(\k)} \ot_{W(\k)}\k \simeq \C$ and $\C_{W(\k)} \ot_{W(\k)}\F$ is a symmetric fusion category
over $\F$. It is easy to see that $\dim(\C_{W(\k)} \ot_{W(\k)}\F)\in W(\k)\subset \F$ and its image
in $\k$ equals $\dim(\C)$.

Thus by \cite[Corollaire 0.8]{Dte} we have an equivalence $\C_{W(\k)} \ot_{W(\k)}\F
\simeq \Rep_\F(G,\eps)$ for a suitable finite group $G$ and central element $\eps \in G$
of order $\le 2$, see Example \ref{gexa} (iii). Since $\dim(\Rep_\F(G,\eps))=|G|$ the non-degeneracy
of $\C$ forces that $|G|$ is not divisible by $p$. In particular, $\eps =1$ if $p=2$ and 
$\Rep_\F(G,\eps)$ is also a lifting of non-degenerate symmetric fusion category
$\Rep_\k(G,\eps)$ for any $p$. By \cite[Theorem 9.6]{ENO} we have an equivalence
of symmetric fusion categories $\C \simeq \Rep_\k(G,\eps)$ and we get the result by
using the forgetful functor $\Rep_\k(G,\eps) \to \sVec$ or $\Rep_\k(G,1)=\Rep_\k(G)\to \Vec$
in the case $p=2$.
\end{proof}

We will need the following criterion of non-degeneracy:

\begin{proposition} \label{semndg}
Let $\C$ be a spherical fusion category such that the ring $K(\C)\ot \k$
is semisimple. Then $\C$ is non-degenerate.
\end{proposition}

\begin{proof} Let $\Tr(x)\in \k$ be the trace of the operator of left multiplication by $x\in K(\C)\ot \k$.
Since $K(\C)\ot k$ is semisimple the {\em trace form} $x,y\mapsto \Tr(xy)$ on $K(\C)\ot k$
is non-degenerate. For $X,Y \in \C$ we have a congruence modulo $p$:
$$\Tr([X][Y])\equiv \sum_{Z\in \O(\C)}\dim \Hom(X\ot Y\ot Z,Z)=\dim \Hom(X\ot Y, \oplus_{Z\in \O(\C)}
Z^*\ot Z).$$

Now consider an element $R=[\oplus_{Z\in \O(\C)}Z^*\ot Z)]\in K(\C)\ot \k$. In the basis
$[X], X\in \O(\C)$ the operator of left multiplication by $R$ has matrix entries 
$$\dim \Hom((\oplus_{Z\in \O(\C)}Z^*\ot Z)\ot X, Y)=\dim \Hom(X\ot Y^*, \oplus_{Z\in \O(\C)}
Z^*\ot Z).$$
Thus the matrix of this operator differs from the matrix of the trace form only by permutations
of columns. Thus under the assumptions of the Proposition this matrix is non-degenerate and
the element $R\in K(\C)\ot \k$ is invertible. Thus its image under the homomorphism 
$\dim : K(\C)\ot \k \to \k$ is nonzero. The result follows since 
$$\dim(R)=\dim (\oplus_{Z\in \O(\C)}Z^*\ot Z)=\sum_{Z\in \O(\C)}\dim(Z)^2=\dim(\C).$$
\end{proof}

\begin{remark} It seems reasonable to expect that conversely for a non-degenerate fusion
category $\C$ the ring $K(\C)\ot \k$ is semisimple.
\end{remark}

The following result is well known. However we did not find a reference, so a proof is included
for reader's convenience.

\begin{lemma}\label{grim}
 Let $\C$ be a faithfully $G-$graded spherical fusion category with neutral
component $\C_1$. Then
$$\dim(\C)=|G|\dim(\C_1).$$
\end{lemma}

\begin{proof} Let $\O_g(\C)\subset \O(\C)$ consists of $X$ with $\phi(X)=g\in G$.
Let
$$D_g=\sum_{X\in \O_g(\C)}\dim(X)[X]\in K(\C)\ot \k.$$
Note that by \ref{ndeg} (a) we have
$D_g\ne 0$ for any $g\in G$. We claim that for $X\in \O_g(\C)$ we have
$$ [X]D_h=\dim(X)D_{gh}\; \mbox{and}\; D_h[X]=\dim(X)D_{hg}.$$
Here is the proof of the first formula (and the second one is similar):

$$[X]D_h=\sum_{Y\in \O_h(\C)}\dim(Y)[X\ot Y]=\sum_{Y\in \O_h(\C), Z\in \O_{gh}(\C)}\dim(Y)
\dim \Hom(X\ot Y,Z)[Z]=$$
$$\sum_{Y\in \O_h(\C), Z\in \O_{gh}(\C)}\dim(Y)
\dim \Hom(Y,X^*\ot Z)[Z]=\sum_{Z\in \O_{gh}(\C)}\dim(X^*\ot Z)[Z]=$$
$$\sum_{Z\in \O_{gh}(\C)}\dim(X^*)\dim(Z)[Z]=\dim(X)D_{gh}.$$
It follows that $D_gD_h=\dim(D_g)D_{gh}=\dim(D_h)D_{gh}$, so $\dim(D_g)=\dim(D_h)$ for
all $g,h\in G$. The result follows since 
$$\dim(\C)=\sum_{g\in G}\dim(D_g)\; \mbox{and}\; \dim(\C_1)=\dim(D_1).$$
\end{proof}

\begin{remark} (i) Argument in the proof of Lemma \ref{grim} is fairly standard, see e.g. \cite[Theorem 3.5.2]{EGNO}.

(ii) Using construction of {\em pivotalization} (see \cite[Definition 7.21.9]{EGNO}) one can
extend Lemma \ref{grim} to fusion categories which are not necessarily spherical.

(iii) In the special case $p=0$ Lemma \ref{grim} is \cite[Corollary 4.28]{DGNO}.
\end{remark}

\subsection{Negligible morphisms} \label{negl}
Let $\C$ be as in the beginning of Section \ref{ndeg} and assume that $\C$ is equipped 
with a spherical structure. We recall that a morphism $f: X\to Y$ in $\C$ is called {\em
negligible} if for any morphism $u: Y\to X$ the trace of the composition $fu$ equals zero,
see e.g. \cite{A+, BW, Ds}. For $X,Y\in \C$ let $\N(X,Y)\subset \Hom(X,Y)$ denote the subspace 
of negligible morphisms.
It is well known that negligible morphisms form a {\em tensor ideal} in $\C$.
This means that a composition $fg$ and tensor product $f\ot g$ is negligible whenever at least
one of $f$ and $g$ is negligible. Thus one defines a new category $\bar \C$ called 
{\em quotient of $\C$ by negligible morphisms} as follows: objects of $\bar \C$ are the same
as objects of $\C$ and $\Hom_{\bar \C}(X,Y)=\Hom_\C(X,Y)/\N(X,Y)$ and the composition of
morphisms in $\bar \C$ is induced by composition in $\C$. We will denote by $\bar X$ an
object of $\bar \C$ corresponding to $X\in \C$.

The tensor product in $\C$
descends to a tensor product in $\bar \C$; thus $\bar \C$ is a tensor category endowed with
a $\k-$linear quotient tensor functor $\C \to \bar \C$ sending $X\in \C$ to $\bar X \in \bar \C$. 
The category $\bar \C$ is equipped with spherical
structure and the quotient functor is compatible with the spherical structures. In addition
the category $\bar \C$ is braided or symmetric if $\C$ is. 

We will use the following result:

\begin{proposition}[\cite{BW} Proposition 3.8, see also \cite{A+} Theorem 2.7 and \cite{EGNO} Exercise 8.18.9] \label{barw}
Assume that $\C$ is abelian and that all
morphism spaces in $\C$ are finite dimensional. Then $\bar \C$ is semisimple and its simple 
objects are precisely $\bar X$ where $X$ is an indecomposable object of $\C$
with $\dim(X)\ne 0$.
\end{proposition}

Note that if $X$ is an indecomposable object with $\dim(X)=0$ then $\id \in \Hom(X,X)$ is
negligible and $\bar X=0$.

\begin{example} \label{produ}
In the setup of Example \ref{equi} consider the quotient $\bar \C_G$ 
of $\C_G$ by the negligible morphisms. The indecomposable objects of 
$\C_G=\C \bt \Rep_\k(G)$ are of the form $X\bt V$ where $X\in \O(\C)$ and $V$ is an
indecomposable object of $\Rep_\k(V)$. We have $\dim(X\bt V)=\dim(X)\dim(V)=0$ if and
only if $\dim(V)=0$, see \ref{ndeg} (a). 
Thus Proposition \ref{barw} combined with Lemma \ref{AA} imply
that $\bar \C_G=\C \bt \overline{\Rep_\k(G)}$ where $\overline{\Rep_\k(G)}$ is the quotient
of $\Rep_\k(G)$ by the negligible morphisms.
\end{example}

\section{Frobenius functor}\label{Frobe}
\subsection{Representations of the cyclic group}\label{green}
Assume that $p>0$. Let $C_p$ be the cyclic group of order $p$ with generator $\sigma$.
Let $\k[C_p]$ be the group algebra of $C_p$; clearly $\k[C_p]=\k[\sigma]/(\sigma^p-1)=
\k[\sigma]/(\sigma-1)^p$.
For any $s\in \BZ$ satisfying $1\le s\le p$ let $L_s$ be $C_p-$module
$\k[\sigma]/(1-\sigma)^s$. Clearly $\dim(L_s)=s$. By the Jordan normal
form theory we have:

(a) $L_s$ exhaust all the isomorphism classes of indecomposable objects 
in the category $\Rep_\k(C_p)$. Moreover, $L_s^*\simeq L_s$.

The decompositions of the tensor products of the modules $L_s$ were described by Green in \cite{G}. We record here some of his results:

(b) $L_1=\be$ is the unit object of $\Rep_\k(C_p)$;
\begin{equation}
L_2\ot L_s=\left\{ \begin{array}{cl} L_2& \mbox{if}\; s=1\\ L_{s-1}\oplus L_{s+1}& \mbox{if}\; s=2,\ldots, p-1\\L_p\oplus L_p& \mbox{if}\; s=p\end{array}\right.
\end{equation}
\begin{equation} \label{smul}
L_{p-1}\ot L_s=L_{p-s}\oplus (s-1)L_p
\end{equation}




\subsection{Universal Verlinde category}

\begin{definition} We define the {\em universal Verlinde category} $\Ver_p$ to be the quotient
of the category $\Rep_\k(C_p)$ by the negligible morphisms. 
\end{definition}

By the results of section \ref{negl}, $\Ver_p$ is a symmetric fusion category. The simple objects
of $\Ver_p$ are precisely $\bar L_s, s=1, \ldots ,p-1$
(note that $\bar L_p=0$). Obviously $\bar L_1$ is the unit object of $\Ver_p$. 
The results of section \ref{green} imply the following relations
in the Grothendieck ring $K(\Ver_p)$:
\begin{equation}\label{vermul}
[\bar L_2][\bar L_s]=[\bar L_{s-1}]+[\bar L_{s+1}], s=1, \ldots ,p-1,
\end{equation}
where we define $[\bar L_0]=[\bar L_p]=0$. Using relation \eqref{vermul} one
determines the multiplication in $K(\Ver_p)$:
\begin{equation}\label{vermul1}
[\bar L_r][\bar L_s]=\sum_{i=1}^c[\bar L_{|r-s|+2i-1}], \mbox{where}\; c=\min(r,s,p-r,p-s) 
\end{equation}

We record the following consequence of \eqref{vermul1}:

\begin{equation}\label{vermul2}
\bL_3\; \mbox{is a summand of}\; \bL_s\ot \bL_s^*\; \mbox{if}\; s\ne 1,p-1.
\end{equation}

Note that the Grothendieck ring of $\Ver_p$ as a ring with basis coincides with so called {\em Verlinde
ring} associated with the quantum group $SL_2$ at $2p-$th root of unity or affine Lie
algebra $\hat sl_2$ at the level $p-2$, see e.g. \cite[4.10.6]{EGNO}. This is our motivation for the choice of the name.

\begin{example}\label{lexam}
(i) The category $Ver_2$ has just one simple object $\bar L_1=\be$ up to isomorphism; thus we have $Ver_2\simeq \Vec$. 

(ii) The category $Ver_3$ has two simple objects $\bar L_1=\be$ and $\bar L_2$ up to isomorphism; 
by \eqref{smul} we have $\bar L_2\ot \bar L_2\simeq \bL_1$; since $\dim(\bL_2)=-1$ we get that $Ver_2\simeq \sVec$.

(iii) The category $Ver_5$ has four simple objects $\bar L_1=\be$, $\bar L_2, \bar L_3, \bar L_4$ 
up to isomorphism; one determines from \eqref{vermul1} that $\bL_4\ot \bL_4\simeq \bL_1=\be$ and
$\bL_3\ot \bL_3\simeq \bL_1\oplus \bL_3$. Thus the fusion subcategory 
$\langle \bL_1, \bL_3\rangle$ is an example of Yang-Lee category from Section \ref{mconj}
\end{example}

It follows from \eqref{smul} that we have
\begin{equation}\label{delta} \bL_{p-1}\ot \bL_s=\bL_{p-s}\end{equation}
In particular we have $\bL_{p-1}\ot \bL_{p-1}\simeq \bL_1=\be$. Thus for $p>2$ the
direct sums of $\bL_1$ and $\bL_{p-1}$ form a fusion subcategory of $\Ver_p$; it is easy
to see that this subcategory is tensor equivalent to $\sVec$ and we will refer to this subcategory
as $\sVec \subset \Ver_p$. On the other hand it follows from \eqref{vermul1} that for $p>3$ the
direct sums of $\bL_1, \bL_3, \ldots, \bL_{p-2}$ also for a fusion subcategory 
$\Ver_p^+\subset \Ver_p$.

\begin{proposition} \label{subV}
Assume $p>3$.

(i) The subcategory $\Ver_p^+\subset \Ver_p$ is generated by $\bL_3$;

(ii) The category $\Ver_p$ has precisely four fusion subcategories:
$\Vec, \sVec$, $\Ver_p^+$, $\Ver_p$;

(iii) We have an equivalence of symmetric fusion categories $\Ver_p\simeq \Ver_p^+\bt \sVec$.
\end{proposition}

\begin{proof} (i) is immediate from \eqref{vermul1}. If a fusion subcategory of $\Ver_p$ contains
$\bL_s$ with $s\ne 1,p-1$ then by \eqref{vermul2} and (i) it contains $\Ver_p^+$. This implies
(ii). Finally (iii) is immediate from Lemma \ref{AA}.
\end{proof}

\subsection{}
Let $\C$ be a semisimple pre-Tannakian category. Let $\C^{(1)}$ be the {\em Frobenius twist}
of $\C$, that is $\C^{(1)}=\C$ as an additive symmetric tensor category with $\k-$linear
structure changed as follows: for $\lambda \in \k$ and a $\C-$morphism $f$ we set
$\lambda \cdot f:=\lambda^pf$. Thus  the Grothendieck ring $K(\C^{(1)})$
is canonically isomorphic to the Grothendieck ring of $\C$.
Since $\lambda \mapsto \lambda^p$ is an automorphism of
$\k$, $\C^{(1)}$ is a Galois conjugate of $\C$. In particular we have $\Ver_p^{(1)}\simeq \Ver_p$ since
$\Ver_p$ is defined over the prime subfield of $\k$.

Consider a functor $P_0: \C \to \C$ sending an object $X$ to $X^{\ot p}$ and a morphism $f:X \to Y$
to $\underbrace{f\ot \ldots \ot f}_{\mbox{$p$ factors}}$.  This functor is not additive but it has an obvious structure of symmetric tensor functor. Moreover, the commutativity isomorphisms 
determine an action of the symmetric group $S_p$ on an object $X^{\ot p}$ (see e.g. \cite[9.9]{EGNO}),
so we can upgrade the functor $P_0$ to the functor taking values in the category of
equivariant objects. We will use only a part of this structure as follows. Let $C_p\subset S_p$
be the cyclic subgroup generated by $p-$cycle $\sigma =(1,2,\ldots,p)$. Restricting the action
of $S_p$ above to $C_p$ we get a symmetric tensor functor $P_1: \C \to \C_{C_p}$ where
$\C_{C_p}$ is the equivariantization of $\C$ as in Example \ref{equi}. 

Let $\bar \C_{C_p}$ be the quotient of $\C_{C_p}$ by the negligible morphisms, see
Section \ref{negl}. Recall that we have an identification $\bar \C_{C_p}=\C \bt \overline{\Rep_\k(C_p)}=\C \bt \Ver_p$, see Example \ref{produ}. Let $Q$ be the quotient functor
$\C_{C_p}\to \bar \C_{C_p}$.
We define a symmetric tensor functor $\Fr_0$ as a composition
$$\C \xrightarrow{P_1}\C_{C_p}\xrightarrow{Q}\bar \C_{C_p}=\C \bt \Ver_p.$$

We have the following

\begin{lemma} The functor $\Fr_0$ is additive.
\end{lemma}

\begin{proof} We have 
$$P_1(f+g)=\underbrace{(f+g)\ot \ldots \ot (f+g)}_{\mbox{$p$ factors}}=P_1(f)+P_1(g)+
\mbox{other terms}$$
where the other terms are monomials $h_1\ot \ldots h_p$ where each $h_i$ is $f$ or $g$ and not
all $h_i$ are the same. The group $C_p$ acts on such monomials by permuting tensorands cyclically; clearly such an action has no fixed points. Thus $P_1(f+g)-P_1(f)-P_1(g)$ splits into summands of the form
\begin{equation}\label{cycsum} 
h_1\ot h_2\ot \ldots \ot h_p+h_2\ot h_3\ot \ldots \ot h_1+\ldots +h_p\ot h_1\ot \ldots \ot h_{p-1}.
\end{equation}
Sum \eqref{cycsum} is a morphism in the category $\C_{C_p}$. Let us show
that this morphism is negligible. Thus we need to show that the trace of the composition 
of \eqref{cycsum} with a suitable morphism $u$ is zero. 

Observe that
$$h_2\ot h_3\ot \ldots \ot h_1=\sigma (h_1\ot h_2\ot \ldots \ot h_p)\sigma^{-1}$$
whence 
$$\Tr ((h_2\ot h_3\ot \ldots \ot h_1)u)=
\Tr((h_1\ot h_2\ot \ldots \ot h_p)\sigma^{-1}u\sigma)=\Tr((h_1\ot h_2\ot \ldots \ot h_p)u)$$
since by the definition of morphisms in $\C_{C_p}$ the morphism $u$ commutes with $\sigma$.
We see that the contribution of each summand in \eqref{cycsum} to the total trace is the same,
which shows that the total trace is zero since we have $p$ summands. 
Hence $\Fr_0(f+g)=\Fr_0(f)+\Fr_0(g)$ as desired.
\end{proof}

The functor $\Fr_0$ is not $\k-$linear since obviously $\Fr_0(\lambda f)=\lambda^p\Fr_0(f)$ for a morphism $f$
and $\lambda \in \k$.

\begin{definition} The {\em Frobenius functor} $\Fr: \C \to \C^{(1)}\bt \Ver_p$ is a $\k-$linear symmetric
tensor functor derived from $\Fr_0$ using the identifications 
$$\C \bt \Ver_p=(\C \bt \Ver_p)^{(1)}=\C^{(1)}\bt \Ver_p^{(1)}=\C^{(1)} \bt \Ver_p.$$
\end{definition}

\begin{example} Let $p=5$ and let $\C$ be the Yang-Lee category from Section \ref{conjsec}. Let us compute $\Fr(X)$. 
We have $X^{\ot 5}=3\cdot \be \oplus 5X$ whence $\Fr(X)=\be \bt \bar V_\be\oplus X\bt \bar V_X$ where $V_\be, V_X\in \Rep_\k(C_p)$ and $\dim(V_\be)=3, \dim(V_X)=5$. The only possibility compatible with $\FPdim(\Fr(X))=\FPdim(X)$ is
$\Fr(X)=\be \bt \bL_3$. Thus $\sigma$ acts on $V_\be$ as a Jordan cell of size 3 and on $V_X$ as a Jordan cell of size 5. 
\end{example}

The following result is follows directly from definitions:

\begin{lemma} \label{stup}
Let $\id \ot \dim$ be the ring homomorphism $K(\C)\ot K(\Ver_p)\to K(\C)\ot \k$ sending $x\ot y$ to $x\ot \dim(y)$.
Then $\id \ot \dim ([\Fr(X)])=[X]^p\in K(\C)\ot \k$.\qed
\end{lemma} 

\begin{example} \label{FrVer}
Let $\C =\Ver_p$ and $p>2$.
Using \eqref{vermul1} one computes $[\bL_2]^p=-2[\bL_{p-1}] \pmod{p}$ in $K(\C)$.
Using Lemma \ref{stup} and the Frobenius-Perron dimension we deduce that 
$\Fr(\bL_2)=\bL_{p-1}\bt \bL_{p-2}\in \C^{(1)}\bt \Ver_p$. Using \eqref{vermul1} again we obtain
$$\Fr(\bL_s)=\left\{ \begin{array}{cc}\be \bt \bL_s&\mbox{if $s$ is odd,}\\ 
\bL_{p-1}\bt \bL_{p-s}&\mbox{if $s$ is even.}\end{array}\right.$$
\end{example}

We will say that $\C$ is of {\em Frobenius type} $\A$ if $\A \subset \Ver_p$ is the smallest fusion subcategory 
such that the image $\Fr(\C)$ is contained in $\C^{(1)}\bt \A \subset \C^{(1)}\bt \Ver_p$. Thus by Proposition \ref{subV} (ii)
for $p>3$ there are just four possibilities for the Frobenius type of $\C$. For example the categories $\Vec, \sVec$ are
of Frobenius type $\Vec$ and the category $\Ver_p$ is of Frobenius type $\Ver_p^+$ by Example \ref{FrVer}. 

\begin{remark} Observe that the formation of the Frobenius functor is compatible with $\k-$linear symmetric tensor functors,
that is for such a functor $F: \C \to \D$ we have $\Fr(F(X))=(F^{(1)}\bt \id)(\Fr(X))\in \D^{(1)}\bt \Ver_p$. It follows
that any category that admits a fiber functor or a super fiber functor is of Frobenius type $\Vec$. Moreover,
Conjecture \ref{mconj} implies that a semisimple pre-Tannakian category $\C$ of subexponential growth is of Frobenius type
$\Vec$ or $\Ver_p^+$.
\end{remark}

In the special case of category $\C$ of Frobenius type
$\Vec$ the image of Frobenius functor is contained in $\C^{(1)}\bt \Vec =\C^{(1)}$. We have the following immediate
consequence of Lemma \ref{stup}:

\begin{corollary}\label{FrFr}
 Let $\C$ be of Frobenius type $\Vec$. Then we have the following equality in $K(\C)\ot \k$:
 \begin{equation*}
[X]^p=[\Fr(X)].\qed
\end{equation*}
\end{corollary}

\section{Proof of Theorem \ref{main}} 
\subsection{Frobenius injective categories}
We say that a symmetric fusion category is {\em Frobenius injective} if the Frobenius functor
is injective. The following result is crucial.

\begin{lemma} \label{nonilp}
 Let $\C$ be a symmetric fusion category which is Frobenius injective of Frobenius
type $\Vec$. Then $\C$ is non-degenerate.
\end{lemma}

\begin{proof} By the assumptions the Frobenius functor $\C \to \C^{(1)}\bt \Ver_p$ lands 
to $\C^{(1)}\bt \Vec =\C^{(1)}$ and is an equivalence. By Corollary \ref{FrFr} we have $[X]^p=[\Fr(X)]$
in $K(\C)\ot \k$. It follows that the linear map $x\mapsto x^p$ on the ring $K(\C)\ot \k$ is
surjective and hence injective. Therefore the commutative ring $K(\C)\ot \k$ has no
nilpotent elements and therfore it is semisimple. The result follows by Proposition \ref{semndg}.
\end{proof} 

\begin{remark} We give here an easy alternative argument in the special case $p=2$. We claim that
in this case a Frobenius injective category $\C$ has no nontrivial self-dual simple objects. Indeed, if $X$
is self-dual then $\be$ appears in $X^{\ot 2}$ with multiplicity 1 and Corollary \ref{FrFr} implies that
$\be$ appears as a direct summand in $\Fr(X)$ which contradicts Frobenius injectivity. This implies Lemma \ref{nonilp}
in this case since contribution of each pair $(X, X^*)$ of non self-dual simple objects to $\dim(\C)$ equals 
$2\dim(X)^2=0$ and hence $\dim(\C)=1$.   
\end{remark}

\begin{corollary} \label{nosVec}
Let $p>2$ and let $\C$ be a symmetric fusion category which is Frobenius 
injective of Frobenius type $\sVec$. Then $\C$ is non-degenerate.
\end{corollary}

\begin{proof} For a simple object $X\in \C$ we have either $\Fr(X)=Y\bt \be$ or $\Fr(X)=
Y\bt \bL_{p-1}$. Let $\phi : \O(\C)\to \BZ/2\BZ$ be the function sending the objects
of the first type to $0\in \BZ/2\BZ$ and the objects of the second type to $1\in \BZ/2\BZ$.
Then $\phi$  is a faithful $\BZ/2\BZ -$grading (see Section \ref{fgen}) of the
category $\C$. Moreover the neutral component $\C_0$ is Frobenius injective of Frobenius type
$\Vec$. Thus $\C_0$ is non-degenerate by Lemma \ref{nonilp}. The result follows since
by Lemma \ref{grim} $\dim(\C)=2\dim(\C_0)$.

\end{proof}

\begin{lemma}\label{Vplus}
 Let $\C$ be of Frobenius type $\Ver_p$ or $\Ver_p^+$ (so $p>3$).
Then the image of Frobenius functor contains $\be \bt \Ver_p^+\subset \C^{(1)}\bt \Ver_p$.
\end{lemma}

\begin{proof} By assumption there is a simple object $X\in \C$ such that $\Fr(X)$ contains
a summand of the form $Y\bt \bL_s$ with $s\ne 0, p-1$. Then $\Fr(X\ot X^*)$ contains
a summand $(Y\ot Y^*)\bt (\bL_s\ot \bL_s^*)$. Since $\be \subset Y\ot Y^*$ and $\bL_3
\subset \bL_s\ot \bL_s^*$ we see that the image of Frobenius functor contains $\be \bt \bL_3$.
The result follows since $\bL_3$ generates $\Ver_p^+$ by Proposition \ref{subV} (i).
\end{proof}

\subsection{Completion of the proof} For a sake of contradiction let us assume that 
Theorem \ref{main} does not hold. Then there exists a counterexample $\C$ with minimal
possible $\FPdim(\C)$, see Section \ref{FPdim}. Then any $\k-$linear symmetric tensor functor from $\C$
to another symmetric fusion category is injective by Corollary \ref{mininj}. In particular, the category $\C$
is Frobenius injective. 

If $\C$ is of Frobenius type $\Vec$ then by Lemma \ref{nonilp} and Proposition \ref{ndgfib} there exists
(necessarily injective) $\k-$linear symmetric tensor functor $\C \to \sVec$ and we have
a contradiction. Similarly, if $\C$ is of Frobenius type $\sVec$ then by Corollary \ref{nosVec} 
and Proposition \ref{ndgfib} there exists $\k-$linear symmetric tensor functor $\C \to \sVec$ and we also have
a contradiction. Note that this completes the proof in the cases $p=2$ and $p=3$.

Thus $\C$ is forced to be of Frobenius type $\Ver_p^+$ or $\Ver_p$. Recall that
$\C$ is Frobenius injective. Let $\tilde \C \subset \C$
be the subcategory generated by simple objects $X$ such that $\Fr(X)=Y\bt \be$ or
$\Fr(X)=Y\bt \bL_{p-1}$. Clearly $\tilde \C$ is a fusion subcategory of $\C$ of Frobenius type
$\Vec$ or $\sVec$. Thus by Lemma \ref{nonilp} and Corollary \ref{nosVec} $\tilde \C$
is non-degenerate. 

\begin{lemma}\label{last}
(i) The image of $\C$ under the Frobenius functor 
is generated by $\Fr(\tilde \C)$ and $\be \bt \Ver_p^+$.

(ii) The fusion subcategory generated by $\Fr(\tilde \C)$ and $\Ver_p^+$ is equivalent to
$\tilde \C \bt \Ver_p^+$.
\end{lemma}

\begin{proof} (i)  We have $\Fr(\C)\supset \Fr(\tilde \C)$; also 
the image $\Fr(\C)$ contains $\be \bt \Ver_p^+$ by Lemma \ref{Vplus}. 
Thus it remains to show that $\Fr(\C)$ is contained in the fusion subcategory
generated by $\Fr(\tilde \C)$ and $\be \bt \Ver_p^+$.

Recall that $\C$ is Frobenius injective.
Let $X\in \O(\C)$ with $\Fr(X)=T\bt \bL_s$. Let $\delta_s=\be$ if $s$ is odd and $\delta_s=\bL_{p-1}$ if $s$ is even.
Then $\bL_s\ot \delta_s\in \Ver_p^+$ (see \eqref{delta}) and there exists
$Y\in \O(\C)$ such that $\Fr(Y)=\be \bt (\bL_s\ot \delta_s)$. 
Then $\Fr(X\ot Y)=T\bt (\bL_s\ot \bL_s\ot \delta_s)$ contains a summand $T\bt \delta_s$. 
Hence $X\ot Y$ contains a summand $Z\in \O(\C)$ such that $\Fr(Z)=T\bt \delta_s$. Thus $Z\in \tilde \C$
and $\Fr(X)=T\bt \bL_s$ is isomorphic to $T\bt (\delta_s \ot \bL_s\ot \delta_s)=\Fr(Z)\ot (\be \bt (\bL_s\ot \delta_s))$ 
which shows that $\Fr(X)$ is contained in the subcategory generated by $\Fr(\tilde \C)$ and $\be \bt \Ver_p^+$ as desired.

(ii) The simple objects of $\Fr(\tilde \C)$ are of the form $T\bt \delta$ where $\delta =\be$ or $\delta =\bL_{p-1}$, 
and the simple objects of $\be \bt \Ver_p^+$ are $\be \bt \bL_s$ with odd $s$. Thus the only simple object which
belongs to both subcategories is $\be \bt \be$ and the result follows from Lemma \ref{AA}. 
\end{proof}

Thus by Lemma \ref{last} the Frobenius functor induces a $\k-$linear symmetric tensor functor
$\C \to \tilde \C \bt \Ver_p^+$. Since $\tilde \C$ is non-degenerate by Proposition \ref{ndgfib} there exists 
a $\k-$linear symmetric tensor functor $\tilde \C \to \sVec$. Taking the composition we get
a functor $\C \to \sVec \bt \Ver_p^+=\Ver_p$ (see Proposition \ref{subV} (iii)). Thus $\C$ is not a counterexample
to Theorem \ref{main}, so no such counterexample exists. 

\subsection{Examples and complements}
\subsubsection{}
Let $p>0$ and
let $G_{a,1}$ be the Frobenius kernel of the additive group $G_a$, see e.g. \cite[2.2]{J}.
Then representations of $G_{a,1}$ are the same as representations of the Hopf algebra $\k [x]/x^p$ where
$x$ is primitive element (that is $\Delta(x)=x\ot 1+1\ot x$). The indecomposable
objects of $\Rep_\k(G_{a,1})$ are Jordan cells of sizes $1,\ldots,p$; moreover the decompositions
of tensor products are precisely the same as in Section \ref{green}, see e.g. \cite[p. 611]{G}.
In particular, the Grothendieck ring of the quotient category $\overline{\Rep_\k(G_{a,1})}$ is
isomorphic to the Verlinde ring $K(\Ver_p)$ as a based ring; moreover the isomorphism respects the dimensions
of objects. We claim that the functor $F: \overline{\Rep_\k(G_{a,1})}\to \Ver_p$ existing 
by Theorem \ref{main} is an equivalence.
Indeed, it is easy to see from explicit formula \cite[Exercise 4.10.7]{EGNO} that for any $s\ne 1,2,p-2,p-1$
we have $\FPdim(\bL_s)>\FPdim(\bL_2)$ and $\FPdim(\bL_2)\not \in \BZ$ for $p>3$.
Hence $F$ should send the two dimensional Jordan cell
to either $\bL_2$ or $\bL_{p-2}$; however the second case is impossible 
since $\dim(\bL_{p-2})=p-2\ne 2$.
Therefore \eqref{vermul} implies that $F$ sends the Jordan cell of size $s$ to $\bL_s$ and thus
$F$ is an equivalence $\overline{\Rep_\k(G_{a,1})}\simeq \Ver_p$.

\subsubsection{}\label{sl2}
In \cite{GK, GM} the authors construct examples of symmetric fusion categories over $\k$ as follows. Let $G$ be a semisimple algebraic group such that its {\em Coxeter number} (see
e.g. \cite[II.6.2]{J}) is
smaller than $p$. The category $\Rep_\k(G)$ contains a Karoubian (but not abelian) tensor 
subcategory $\T(G)$ of {\em tilting modules},
see \cite{A,GM} or \cite[II.E]{J}. The quotient $\overline{\T(G)}$ of this category by the negligible morphisms (see Section \ref{negl}) is an example of symmetric fusion category.
In the special case $G=SL_2$ it is known
that $K(\overline{\T(SL_2)})\simeq K(\Ver_p)$ as a based ring and the isomorphism 
respects the dimensions, see \cite{GM}.
As in the preceding paragraph it follows that we have an equivalence of symmetric fusion categories
$\overline{\T(SL_2)}\simeq K(\Ver_p)$ as it was promised in Section \ref{conjsec}.

\subsubsection{}
Finally, we sketch a direct construction of the functor $\overline{\T(G)}\to \Ver_p$ guaranteed by
Theorem \ref{main}. Let $G_{a,1}\subset G$ be an embedding associated with a {\em regular
nilpotent element} of the Lie algebra of $G$. The restriction gives a $\k-$linear symmetric
tensor functor $\T(G)\to \Rep_\k(G_{a,1})$. The theory of {\em support varieties} shows that
an indecomposable object of $\T(G)$ of dimension zero is sent by this functor to a projective
object of $\Rep_\k(G_{a,1})$, see \cite[E13]{J}. It follows that the restriction functor descends
to a functor $\overline{\T(G)}\to \overline{\Rep_\k(G_{a,1})}$. Combining this with equivalence
$\overline{\Rep_\k(G_{a,1})}\simeq \Ver_p$ we get a desired 
functor $\overline{\Rep_\k(G_{a,1})}\to \Ver_p$.


\bibliographystyle{ams-alpha}

\end{document}